\documentclass{amsart}
\usepackage{xypic}

\xyoption{all}

\newif\iffurther

\newtheorem{thm}{Theorem}[section] 

\newtheorem{cor}[thm]{Corollary}

\newtheorem{lem}[thm]{Lemma}
\newtheorem{prop}[thm]{Proposition}

\def\[{\left[}
\def\]{\right]}

\def\Span{{\operatorname{Span}}}

\def\GKdim{{Gel'fand-Kirillov dimension}}

\long\def\forget#1\forgotten{{}}

\begin{document}

\title[]{Growth of monomial algebras, simple rings and free subalgebras}

\author{Be'eri Greenfeld}
\address{Department of Mathematics, Bar Ilan University, Ramat Gan 5290002, Israel}
\email{beeri.greenfeld@gmail.com}

\thanks{The author thanks the referee for his/her useful comments on the paper.
}

\date{\today}


\begin{abstract}
We construct finitely generated simple algebras with prescribed growth types, which can be arbitrarily taken from a large variety of (super-polynomial) growth types. This (partially) answers a question raised in \cite[Question~5.1]{Greenfeld}.

Our construction goes through a construction of finitely generated just-infinite, primitive monomial algebras with prescribed growth type, from which we construct uniformly recurrent infinite words with subword complexity having the same growth type.

We also discuss the connection between entropy of algebras and their homomorphic images, as well as the degrees of their generators of free subalgebras.
\end{abstract}

\maketitle

\setcounter{tocdepth}{3}

\section{Introduction}

Fix a function $f:\mathbb{N}\rightarrow \mathbb{N}$ satisfying the following properties, which every growth function of a finitely generated associative algebra satisfies:
\begin{itemize}
\item Monotonely increasing, namely $f(n)<f(n+1)$ for all $n\in \mathbb{N}$;
\item Submultiplicativity, namely $f(n+m)\leq f(n)f(m)$ for all $n,m\in \mathbb{N}$.
\end{itemize}

We think of two such functions $f,g$ as representing the same growth type, denoted by $f\sim g$ if there exist constants $C,D>0$ for which $f(n)\leq g(Cn)\leq f(Dn)$ for all $n\in \mathbb{N}$.  This notion of equivalence is indeed reasonable since growth functions of finitely generated associative algebras are defined up to this equivalence relation.

In \cite{BartholdiSmoktunowicz}, it is shown how to realize such $f$ as the growth type of some finitely generated monomial algebra, provided that $f$ grows rapidly enough, namely, there is $\alpha>0$ such that $nf(n)\leq f(\alpha n)$ for all $n\in \mathbb{N}$.

In \cite{Greenfeld}, it is shown how to modify the construction of Bartholdi and Smoktunowicz to result in a primitive algebra (without further limitations on the realized function $f$). It is also asked in \cite[Question~5.1]{Greenfeld} whether similar result could be proven for simple algebras.

In this paper we show how the basic construction from \cite{BartholdiSmoktunowicz} can be modified to result in just-infinite algebras, namely, infintie dimensional algebras all of whose proper quotients are finite-dimensional. Since in this context it can be easily shown that just-infinitude implies primitivity, this provides another approach for the main result of \cite{Greenfeld}. These algebras are determined by infinite words, and therefore our algebras in this paper give rise to infinite words which are uniformly recurrent and have prescribed complexity growth type (which is arbitrary up to several mild constraints, see Theorem \ref{uniformly recurrent}). 

Uniformly recurrent words have great importance from combinatorial, dynamical and algebraic points of view \cite{AS, BM, BR, L, F}.
It should be mentioned that in \cite{Julien}, Cassaigne constructs uniformly recurrent words with a large class of intermediate growth functions; the class of functions realized here is not contained in the class realized in \cite{Julien}.

The main application of our new construction comes in the view of a recent construction of Nekrashevych \cite{Nekrashevych}. Namely, the new growth types realized here as complexity functions of uniformly recurrent words allow us to construct finitely generated simple algebras with a larg variety of prescribed growth types, consisting of almost all functions realized in \cite{BartholdiSmoktunowicz} as growth types of algebras (and later in \cite{Greenfeld} as growth types of prime and primitive algebras). This provides a partial answer to \cite[Question~5.1]{Greenfeld}.

Recall that for an infintie word $w$ with letters $x_1,\dots,x_d$ we have a corresponding monomial algebra $A_w$ generated by $x_1,\dots,x_d$, which is spanned by all monomials which occur as subwords of $w$. Vice versa, prime (finitely generated) monomial algebras are induced from infinite words, and many combinatorial properties of $w$ reflect algebraically in $A_w$.
The connection between the asymptotic subword complexity $p_w(n)$ of $w$ and the growth $g_{A_w}(n)$ of its corresponding monomial algebra $A_w$ is intimate: $g_{A_w}(n)=p_w(1)+\cdots+p_w(n)$ where the generating vector space is taken to be $V=\Span\{1,x_1,\dots,x_d\}$.

In the current construction, slight limitations must be put on the function $f$; however it still enables one realize arbitrary super-\GKdim\ (i.e.~$f(n)\sim \exp{n^r}$ for arbitrary $r>0$) or entropy, a notion which we now quickly recall.

Fix a base field $k$. Recall that for a graded $k$-algebra $R=\bigoplus_{i\in \mathbb{N}}R_i$ where $\dim_k R_i<\infty$ the \textit{entropy}, in the sense of Newman, Schneider and Shalev \cite{Entropy} is defined to be: $$H(R) = \limsup_{n\rightarrow \infty} \sqrt[n]{\dim_k {\bigoplus_{i\leq n} R_i}}$$

In \cite[Section~3]{Entropy} it is shown that while the entropy of a free $d$-generated associative algebra is $d$, all of its proper graded quotients have entropy strictly less than $d$. As demonstrated there, it is not true in general that the entropy must reduce when passing to proper quotient algebras. The writers ask if this phenomenon has a unified explanation: `It would be interesting to characterize graded associative algebras whose entropy is larger than that of each of their proper quotients.'

Let us say that an algebra all of whose proper quotients have strictly smaller entropy has \textit{sharp entropy}. Attempting to indicate the difficulty of the problem cited above, we mention that the algebras we construct here, which, as mentioned before have entropy strictly larger than $1$ but all proper quotients finite-dimensional, in particular have sharp entropy. They can be also modified to have relatively arbitrarily small entropy larger than $1$.

In \cite{Smoktunowicz}, Smoktunowicz constructs graded algebras with prescribed relations, having sufficient control on the asymptotic growth of the constructed algebras. In particular, it is noted in \cite{Smoktunowicz} that by \cite[Theorem~1.1]{Smoktunowicz} there exist graded algebras with arbitrarily small entropy.

We mention that in \cite[Corollary~2.5]{Greenfeld}, we show how the mentioned above technique presented in \cite{BartholdiSmoktunowicz} can be modified to result in \textit{prime} algebras with arbitrarily small entropy.

It is straight-forward that if $R$ contains a free subalgebra then $H(R)>1$. There is moreover an evident better lower bound on $R$ in terms of the degree of the homogeneous components on which the free generators are supported (Proposition \ref{exponential}). In this paper we also show that this evident bound cannot be significantly improved: there exists a family of graded algebras with entropy arbitrarily small and free generators in the expected degrees.

In the rest of the paper $k$ is an arbitrary base field, and monomial algebras are endowed with natural grading: $\deg(x_i)=1$ for each letter $x_i$, and we write $R=\bigoplus_n R_n$ as a decomposition into homogeneous components.

\section{Growth of monomial algebras and simple algebras}

We start by a quick overview of the techinque, which was introduced in \cite{BartholdiSmoktunowicz}, where additional details appear.

Let $f:\mathbb{N}\rightarrow \mathbb{N}$ be a monotonely increasing and submultiplicative (i.e. $f(m+n)\leq f(m)f(n)$) function.

We fix letters $x_1,\dots,x_d$ where $d=f(1)$ and define a system of sets $W(2^i)$ of length $2^i$ monomials.

Let $W(1)=\{x_1,\dots,x_d\}$. Suppose that for every $i\in \mathbb{N}$ we have a set $C(2^i)\subseteq W(2^i)$ of cardinality $\lceil \frac{f(2^{i+1})}{f(2^i)} \rceil$ and inductively define $W(2^{i+1})=C(2^i)W(2^i)$. For two monomials write $u\preceq w$ if $u$ is a subword of $w$. Set $W = \bigcup_{i\in \mathbb{N}}W(2^i)$ and define an ideal $I\triangleleft k\left<x_1,\dots,x_d\right>$ generated by all monomials which do not occur as subwords of any monomial from $W$. Finally, define an algebra $R = k\left<x_1,\dots,x_d\right>/I$. As shown in \cite[Corollary~D]{BartholdiSmoktunowicz}, $\dim_k R_n\sim f$. Moreover, as shown in \cite[Theorem~C]{BartholdiSmoktunowicz}, we actually have: $f(2^n)\leq \dim_k R_{2^n}\leq 2^{2n+3}f(2^{n+1})$.

\begin{thm} \label{uniformly recurrent}
Let $f:\mathbb{N}\rightarrow \mathbb{N}$ be a monotonely increasing function such that:
\begin{itemize}
\item There exists some $\alpha>0$ for which $nf(n)\leq f(\alpha n)$ for all $n\in \mathbb{N}$;
\item For all $\beta>0$ there is some $n_\beta\in \mathbb{N}$ such that for all $n\geq n_\beta$ we have: $f(2^{n+1})\leq \frac{1}{\beta}f(2^n)^2$;
\item $\prod_{n=0}^{\infty} \left(1+\frac{f(2^n)}{f(2^{n+1})}\right) < \infty$.
\end{itemize}
Then there exists a fintiely generated monomial, just-infinite primitive algebra whose growth type is $\sim f$. Moreover, this algebra is of the form of the algebras constructed in \cite{BartholdiSmoktunowicz}.
\end{thm}

We now prove a supporting lemma, after which we turn to prove Theorem \ref{uniformly recurrent}.

Recall that an infinite word is \textit{uniformly recurrent} if for every factor $w$ of the infinite word there exists some constant $c=c_w$ such that the next occurence of $w$ in the infinite word never takes longer than $c_w$ letters.

\begin{lem} \label{capture at finite}
Let $R$ be an algebra constructed from a system of sets $W(2^n),C(2^n)$ as described above.
Suppose for every $w$ a non-zero monomial in $R$ there exists some $n=n_w$ for which all monomials in $C(2^n)$ contain $w$ as a factor.

Then $R$ is defined by an infinite uniformly recurrent word.
\end{lem}

\begin{proof}
Let $w$ be a non-zero monomial of $R$. We claim that for $c_w=2^{n_w+1}$ we have that for every non-zero monomial $u$ of $R$ of length at least $c_w$, $w$ occurs in $u$ and the next occurence of $w$ in $u$ never takes longer than $c_w$ letters.

Pick a word $u$ of length at least $2^{n_w}$, and write $|u|=2^{n_w}t+s$ where $s<2^{n_w}$. Now write $u=u_1\cdots u_tu'$ where $|u_i|=2^{n_w}$ and $|u'|=s$.
As mentioned in \cite{BartholdiSmoktunowicz} it follows that $u_i\in C(2^{n_w})W(2^{n_w})\cup W(2^{n_w})C(2^{n_w})$ for all $i$ and in both cases $u_i$ contains $w$ as a subword, hence $w$ occurs in $u$ with no more than $2^{n_w+1}$ letters between each two occurences.

Now $R$ is prime. To see this, fix two monomials $w_1,w_2$ and find a long enough non-zero monomial; then both $w_1,w_2$ occur as subwords of it in disjoint places, and hence $w_1Rw_2\neq 0$. It follows that we can write $R$ as the monomial algebra associated with some infinite word. It is now clear why this word is uniformly recurrent.
\end{proof}

We are now ready to prove Theorem \ref{uniformly recurrent}.

\begin{proof}(Theorem \ref{uniformly recurrent})
We use the technique introduced above, originally from \cite{BartholdiSmoktunowicz}.

To construct sets $W(2^i),C(2^i)$ we need a preliminary tool. We fix a function $\mu:\mathbb{N}\rightarrow \mathbb{N}$ such that $\mu(t)>t+100$ and for all $n\geq \mu(t)$: $$\frac{f(2^{n+1})}{f(2^{n})}\leq \prod_{i=t}^{n-1} \lceil \frac{f(2^{i+1})}{f(2^{i})}\rceil.$$ We claim that this is possible by the second assumption of the Theorem, as we now explain.

Let $\beta = \prod_{i=0}^{t-1} \lceil \frac{f(2^{i+1})}{f(2^{i})} \rceil$. Then the inequality above takes the form $f(2^{\mu(t)+1})\leq \frac{1}{\beta\beta'} f(2^{\mu(t)})^2$ where $\beta'$ is some constant bounding above the following product: 
$$\prod_{i=0}^{\infty} \left(\lceil\frac{f(2^{i+1})}{f(2^i)}\rceil \bigg/ \frac{f(2^{i+1})}{f(2^i)}\right) \leq \prod_{i=0}^{\infty} \left(1+\frac{f(2^i)}{f(2^{i+1})}\right)<\infty$$
where the second inequality follows from the third assumption.

We are now ready to construct sets $W(2^i),C(2^i)$. We work by induction. Set $W(1)=\{x_1,\dots,x_d\}$.

We work by induction; fix $t$ and let $m$ be the least number such that we have already defined all sets $W(2^i),C(2^i)$ for $i\leq m$. Let $w$ be a monomial appearing as a subword of some $w'\in W(2^t)$. Pick some $t'=\max(\mu(t),m+1)$ and for all $i<t'$ if $C(2^i)$ was not yet defined, take it to be arbitrary (up to the trivial conditions $C(2^i)\subseteq W(2^i)$ and on the size of $C(2^i)$).

Now $t'\geq \mu(t)$ so by definition of $\mu$ it follows that we can fix the $2^t$-length suffix of all monomials of $C(2^{t'})$, namely, since $\frac{f(2^{n+1})}{f(2^{n})}\leq \prod_{i=t}^{n-1} \lceil \frac{f(2^{i+1})}{f(2^{i})}\rceil$ we can take all monomials of $C(2^{t'})$ to be of the form $vw'$ where $v$ is arbitrary and $w'$ is our fixed length $2^t$ monomial containing $w$ as a factor.

In this way we are able to define all sets $W(2^i),C(2^i)$, taking care each time of different monomial, covering all non-zero monomials of the algebra, that is: for each non-zero monomial there is some constant for which the conditions of Lemma \ref{capture at finite} are satisfied. It follows that we construct an algebra $R$ with growth $\sim f$ which is defined by a uniformly recurrent infinite word.

Note that by \cite[Theorem~3.2]{Belovetal}, it follows that $R$ is just-infinite. To see that $R$ is moreover primitive apply the following argument.

Since $R$ is a just infinite algebra, $R$ is prime \cite{FP}. Then by a result of Okni\'nski \cite{Trichotomy}, we have that $R$ is either primitive, satisfies a polynomial identity, or has nonzero locally nilpotent Jacobson radical.

However, since the growth type $f$ is super-polynomial it is impossible that $R$ is PI, so suppose it has non-zero Jacobson radical $0\neq J\triangleleft R$. 

Note that $R\neq J$ since by \cite{BeidarFong} the Jacobson radical of a monomial algebra is locally nilpotent. Then $R/J$ is finite-dimensional, non-zero semiprimitive ring, and since $J$ is homogeneous (with respect to the standard grading on $R$) by \cite{Bergman} it follows that $R/J$ is a graded finite-dimensional algebra, hence nilpotent, contradicting that $R/J$ is semiprimitive.
\end{proof}

\subsection{Simple algebras}

As a consequence of Theorem \ref{uniformly recurrent} and Lemma \ref{capture at finite}, combined with the work from \cite{Nekrashevych} we have the following Corollary:

\begin{thm} \label{simple}
For any function $f$ satisfying the conditions of Theorem \ref{uniformly recurrent} there exists a finitely generated simple algebra with growth type $\sim f$.
\end{thm}

This provides a partial answer to \cite[Question~5.1]{Greenfeld}. In \cite{Nekrashevych}, Nekrashevych introduces a method for associating to a Hausdorff \'etale groupoid $\mathcal{G}$ an algebra $k[\mathcal{G}]$ consisting of continuous, compactly supported functions $f:\mathcal{G}\rightarrow k$ (where $k$ is endowed with the discrete topology). As shown in the mentioned paper, the algebras associated with \'etale groupoids constructed out of infinite words (through subshifts) conatin finitely generated simple subalgebras.

\begin{proof}(Theorem \ref{simple})
By Theorem \ref{uniformly recurrent} and Lemma \ref{capture at finite} there exists an infinite uniformly recurrent word $w$ in finitely many letters (say $X=\{x_1,\dots,x_n\}$) having subword complexity $p_w(n)\sim f$ (recall that $f(n)\sim nf(n)$). In the terminology of \cite{Nekrashevych}, $w\in X^{\mathbb{Z}}$ is minimal and non-periodic, hence by \cite[Theorem~1.2]{Nekrashevych} the algebra $\mathcal{A}_w$ is finitely generated, simple and has growth rate $\sim np_w(n)\sim f$.
\end{proof}

\section{Entropy of algebras and free subalgebras}

In \cite{Entropy}, Newman, Schnider and Shalev introduced the notion of \textit{entropy} of a graded algebra, and investigated it from various poins of view, with emphasis on free algebras and their subalgebras (namely, Section $4$ in \cite{Entropy}).

\subsection{Sharp entropy}
In \cite{Entropy}, the authors ask the following: `It would be interesting to characterize graded associative algebras whose entropy is larger than that of each of their proper quotients.'

Let us say that an algebra all of whose proper quotients have strictly smaller entropy has \textit{sharp entropy}. Attempting to indicate the difficulty of the problem cited above, we now show how Theorem \ref{uniformly recurrent} results in algebras with sharp entropy, which is arbitrarily close to $1$. Note that in general, the equivalence relation between growth types (namely, up to constant scaling of the argument) does not preserve the notion of entropy. However, careful analysis of the construction enables to bound the entropy of the resulting algebras, both from below and above.

\begin{cor}
Let $0<\varepsilon<1$. Then there exists a graded algebra $R$ with $1+\frac{1}{3}\varepsilon\leq H(R)\leq 1+3\varepsilon$ with sharp entropy.
\end{cor}

\begin{proof}
We apply Theorem \ref{uniformly recurrent} to the function $f(n)=\lceil \frac{1}{n} (1+\varepsilon)^n \rceil$. Note that it does satisfy the Theorem's assumptions, since $f(2^{n+1})=\lceil \frac{1}{2^{n+1}} (1+\varepsilon)^{2^{n+1}} \rceil$ while $f(2^n)^2={\lceil \frac{1}{n} (1+\varepsilon)^{2^n}\rceil}^2\approx \frac{1}{n^2} (1+\varepsilon)^{2^{n+1}}$, and $2^n$ grows much faster than $n^2$.

In addition, $\frac{f(2^n)n^2}{f(2^{n+1})}\approx 2n^2(1+\varepsilon)^{-2^n}\rightarrow 0$ so also the third assumption holds.

Also note that the entropy associated with $f$ is $1+\varepsilon$. Note that though in general the entropy is not an invariant of the growth type (namely, it might change under multiplying $n$ by some constant), it follows from the proof of \cite[Corollary~2.5]{Greenfeld} that the entropy of the constructed algebra (as it fits into the form of algebras constructed in \cite{BartholdiSmoktunowicz}) is between $\sqrt{1+\varepsilon}$ to $(1+\varepsilon)^2$, and the claim follows.
\end{proof}

\subsection{Free subalgebras}
It is clear that if $R$ contains a free subalgebra then it grows exponentially, namely, $H(R)>1$. More precisely, the following holds:

\begin{prop} \label{exponential}
Let $R=\bigoplus_{i\in \mathbb{N}} R_i$, and let the entropy of $R$ be $H(R)=1+\varepsilon$. Suppose $a,b\in \bigoplus_{i\leq d} R_i$ generate a free subalgebra.

Then: $$d\geq \frac{1}{\log_2(1+\varepsilon)}.$$
\end{prop}

\begin{proof}
Let $S_n=\{w(x,y)|\deg(w)\leq n\}$ be the set of monomials in the non-commuting variables $x,y$ of degree at most $n$. Then $\{w(a,b)|w\in S_n\}$ is a linealry independent subset of $\bigoplus_{i\leq dn} R_i$ of size $2^{n+1}$. Therefore $2^{n+1}\leq \dim_k \bigoplus_{i\leq dn} R_i$, so $2^{\frac{1}{d}}\leq 2^{\frac{n+1}{dn}}\leq H(R)=1+\varepsilon$, and the claim follows.
\end{proof}

We use the technique from \cite{BartholdiSmoktunowicz} used in the previous sections to construct monomial algebras with arbitrarily small entropy and free subalgebras generated by monomials of degrees which are optimal from the point of view of Proposition \ref{exponential}.

\begin{thm}
There exist finitely generated monomial algebras with $H(R)\rightarrow 1^{+}$ and with free subalgebras generated by monomials in degree $\sim \mathcal{O}\left(\frac{1}{\log_2 H(R)}\right)$.
\end{thm}

\begin{proof}
Suppose $f(n)=\lceil (1+\varepsilon)^n \rceil$. As follows from the proof of \cite[Corollary~2.5]{Greenfeld}, we have that $\sqrt{1+\varepsilon}\leq H(R)\leq (1+\varepsilon)^2$ for $R$ constructed as above. From now on we let $\varepsilon<1$, then we get $1+\frac{1}{3}\varepsilon\leq H(R)\leq 1+3\varepsilon$.

We now turn to construct specific sets $W(2^i),C(2^i)$, such that the resulting algebra $R$ has a free subalgebra generated by monomials of specific degree. Fix $t=\lceil -\log_2\log_2(1+\varepsilon) \rceil + 1$, so that $2^t>\frac{1}{\log_2(1+\varepsilon)}=\log_{1+\varepsilon} 2$.

Now $d=\lceil 1+\varepsilon \rceil=2$ and let $x_1,x_2=x,y$. For $i\leq t$ we let $W(2^i),C(2^i)$ be arbitrary, up to the requirement that $x^{2^i},y^{2^i}\in C(2^i)$ which can be clearly be satisfied (as $f(n)\geq 2$ always).

We now turn to define $C(2^{t+r})$ where $r>0$. Note that $$|C(2^{t+r})|=\lceil (1+\varepsilon)^{2^{t+r+1}-2^{t+r}}\rceil=\lceil (1+\varepsilon)^{2^{t+r}}\rceil \geq ((1+\varepsilon)^{2^t})^{2^r}\geq 2^{2^r},$$ and by induction all monomials of length at most $2^{r-1}$ in $x^{2^t},y^{2^t}$ occur as subwords of $W(2^{t+r})$, so we can pick $C(2^{t+r})$ to be an arbitrary subset of $W(2^{t+r})$ containing all products of the form $uv$ where $u,v$ are length $2^{r-1}$ monomials in $x^{2^t},y^{2^t}$; such products are precisely all monomials of length $2^r$ monomials in $x^{2^t},y^{2^t}$, and since there are $2^{2^r}$ of them (and $|C(2^{t+r})|$ is required by construction to have cardinality which is large enough as we saw above) we can proceed the induction step.

Finally, we get an algebra $R$ with entropy $1+\frac{1}{3}\varepsilon\leq H(R)\leq 1+3\varepsilon$ and such that $x^{2^t},y^{2^t}$ generate a free subalgebra (since there is no monomial relation upon them, and $R$ is a monomial algebra). Note that $\deg(x^{2^t})=\deg(y^{2^t})=2^t=\mathcal{O}(\frac{1}{\log_2(1+\varepsilon)})$. Thus, comparing with Proposition \ref{exponential}, we get asymptotic optimality both for the construction and for the Proposition (up to a universal constant multiple).
\end{proof}


\begin{thebibliography}{99}

\bibitem{AS}
J.~Allouche, J.~Shallit, \textit{Automatic Sequences: Theory, Applications, Generalizations}, Cambridge University Press (2013).

\bibitem{BeidarFong}
K.~I.~Beidar, Y.~Fong, \textit{On radicals of monomial algebras}, Communications in Algebra (1998) 26, 12, 3913--3919.

\bibitem{BM}
J.~P.~Bell, B.~Madill, \textit{Iterative algebras}, Algebras and Representation Theory (2015) 18, 6, 1533--1546.

\bibitem{Belovetal}
A.~Ya.~Belov, V.~V.~Borisenko, and V.~N.~Latyshev, \textit{Monomial algebras}, J.~Math.~Sci.~(New York) 87 (1997), no.~3, 3463–3575. Algebra, 4.

\bibitem{Bergman} 
G.~M.~Bergman, \textit{On Jacobson radicals of graded rings}, preprint, 1975. Available at {\tt http://math.berkeley.edu/$\sim$gbergman/papers/unpub/}.

\bibitem{BR}
V.~Berth\'e, M.~Rigo, \textit{Combinatorics, automata, and number theory}, Encyclopedia of Mathematics and its Applications. 135. Cambridge University Press (2010).

\bibitem{Julien}
J.~Cassaigne, \textit{Constructing Inﬁnite Words of Intermediate Complexity}, Chapter `Developments in Language Theory', volume 2450 of the series `Lecture Notes in Computer Science': 173--184.

\bibitem{FP}
J.~Farina and C.~Pendergrass-Ricea, \textit{A few properties of just infinite algebras}, Communications in Algebra {\bf 35} (5) (2007), 1703--1707.

\bibitem{Greenfeld}
B.~Greenfeld, \textit{Prime and primitive algebras with prescribed growth types}, accepted to the Israel Journal of Mathematics, available in arXiv:1604.07477.

\bibitem{L}
M.~Lothaire, \textit{Algebraic combinatorics on words}, Encyclopedia of Mathematics and Its Applications. 90. With preface by Jean Berstel and Dominique Perrin (Reprint of the 2002 hardback ed.). Cambridge University Press (2011).

\bibitem{Nekrashevych}
V.~Nekrashevych, \textit{Growth of \'etale Groupoids and Simple Algebras}, International Journal of Algebra and Computation, 26, 375 (2016).

\bibitem{Entropy}
M.~F.~Newman, C.~Schneider, A.~Shalev, \textit{The entropy of graded algebras}, Journal of Algebra 223, 85–100 (2000).

\bibitem{Trichotomy}
J.~Okni\'nski, \textit{Trichotomy for finitely generated monomial algebras}, Journal of Algebra {\bf 417}, 1 (2014), 145--147.

\bibitem{F}
N.~Pytheas Fogg, \textit{Substitutions in dynamics, arithmetics and combinatorics}, Lecture Notes in Mathematics. 1794. Editors Berthé, Valérie; Ferenczi, Sébastien; Mauduit, Christian; Siegel, A. Berlin: Springer-Verlag (2002).

\bibitem{Smoktunowicz}
A.~Smoktunowicz, \textit{Growth, Entropy and Commutativity of Algebras Satisfying Prescribed Relations}, Selecta Mathematica (2014) 20, 4: 1197--1212.

\bibitem{BartholdiSmoktunowicz}
A.~Smoktunowicz, L.~Bartholdi, \textit{Images of Golod-Shafarevich Algebras with Small Growth}, Quarterly Journal of Mathematics (2014) 65 (2): 421--438.


\end{thebibliography}
\end{document}